\documentclass[11pt,a4paper]{amsart}
\usepackage{amssymb}
\usepackage{mathabx}
\usepackage{mathrsfs}
\usepackage{syntonly}
\usepackage{amsmath}
\usepackage{amsthm}
\usepackage{amsfonts}
\usepackage{amssymb}
\usepackage{latexsym}
\usepackage{amscd,amssymb,amsopn,amsmath,amsthm,graphics,amsfonts,mathrsfs,accents,enumerate,verbatim,calc}
\usepackage[dvips]{graphicx}
\usepackage[colorlinks=true,linkcolor=red,citecolor=blue]{hyperref}
\usepackage[all]{xy}

\date{}
\allowdisplaybreaks[4] \footskip=15pt
\renewcommand{\uppercasenonmath}[1]{}

\numberwithin{equation}{section} \theoremstyle{plain}
\newtheorem{lem}{Lemma}[section]
\newtheorem{cor}[lem]{Corollary}
\newtheorem{prop}[lem]{Proposition}
\newtheorem{thm}[lem]{Theorem}

\newtheorem{definition}[lem]{Definition}
\newtheorem{Ex}[lem]{Example}
\newtheorem{Quest}[lem]{Question}
\newtheorem{Property}[lem]{Property}
\newtheorem{Properties}[lem]{Properties}
\newtheorem{Subprops}{}[lem]
\newtheorem{Para}[lem]{}

\newtheorem{rem}[lem]{Remark}

\newenvironment{df}{\begin{definition}\rm}{\end{definition}}

\newtheorem*{ack*}{ACKNOWLEDGEMENTS}

\newcommand{\pf}{\noindent\begin {proof}}
\newcommand{\epf}{\end{proof}}

\newcommand{\X}{\mathcal{X}}
\newcommand{\W}{\mathcal{W}}

\newcommand{\C}{\mathcal{C}}
\newcommand{\HH}{\mathcal{H}}

\newcommand{\E}{\mathbb{E}}

\newcommand{\oext}{\mbox{\rm Ext}}

\newcommand{\h}{{\rm H}}

\begin{document}
\begin{center}
{\large  \bf  Auslander-Buchweitz Approximation Theory for Extriangulated Categories}

\vspace{0.5cm}  Yajun Ma, Nanqing Ding and Yafeng Zhang

\hspace{-4mm}Department of Mathematics, Nanjing University, Nanjing 210093, China \\
E-mails: 13919042158@163.com, nqding@nju.edu.cn and 470985396@163.com\\
\end{center}

\bigskip
\centerline { \bf  Abstract}
\medskip

\leftskip10truemm \rightskip10truemm \noindent Extriangulated categories were introduced by Nakaoka and Palu as a simultaneous generalization of exact categories and triangulated categories. In this paper, we introduce and develop an analogous theory of Auslander-Buchweitz approximations for extriangulated categories. We establish the existence of precovers $($and preenvelopes$)$ and obtain characterizations of relative homological dimensions, which are based on certain subcategories under finiteness of resolutions. Finally, we give a description of cotorsion pairs on extriangulated categories under some conditions, and provide a characterization of silting subcategories on stable categories. \\[2mm]

\leftskip0truemm \rightskip0truemm
\section { \bf Introduction}


Originated from the concept of injective envelopes, the approximation theory has attracted increasing interest and , hence, obtained considerable development especially in the context of module categories (see, for example \cite{AR,EJ2}). Auslander and Buchweitz \cite{AB} studied the ideals of injective envelopes and projective covers in terms of maximal Cohen-Macaulay approximations for certain modules. Indeed, they established their theory in the context of abelian categories and provided important applications. Inspired by their work, Mendoza Hern$\acute{a}$ndez, S$\acute{a}$enz Valadez, Santiago Vargas and Souta Salorio developed in \cite{MSSS1, MSSS2} an analogous theory of approximations for triangulated categories.

Triangulated categories and exact categories are two fundamental structures in mathematics. They are also important tools in many mathematical branches. It is well known that these two kinds of categories have some similarities, there are even direct connections between them. By extracting the similarities between triangulated categories and exact categories, Nakaoka and palu \cite{NP} recently introduced the notion of extriangulated categories, whose extriangulated structures are given by $\E$-triangles with some axioms. Except triangulated categories and exact categories, there are many examples for extriangulated categories \cite{NP,ZZ}. Hence, many results known to hold on exact categories and triangulated categories can be unified in the same framework \cite{HZZ,LN,NP,ZZ}. Motivated by the ideal, we introduce and develop an analogous theory of approximations in the sense of Auslander and Buchweitz \cite{AB} for extriangulated categories. Let $\C$ be an extriangulated category with enough projectives and injectives. The main results deal with a pair ($\X, \W$) of subcategories of $\C$, where $\X$ is closed under extensions and $\W$ is an $\X$-injective cogenerator for $\X$. We consider the subcategory $\widehat{\X}$ of $\C$ consisting of all objects with a finite resolution by objects of $\X$. Moreover, a notion of $\X$-resolution dimensions is also introduced, which is compared with other relative homological dimensions. We prove that any object of $\widehat{\X}$ admits two $\E$-triangles: one giving rise to an $\X$-precover and the other to a $\widehat{\W}$-preenvelope, which is used to construct a cotorsion pair on the extriangulated category $\widehat{\X}$. Whenever $\C$ is a Frobenius extriangulated category, we give a characterization of hereditary cotorsion pairs $(\mathcal{U}, \mathcal{V})$ on the extriangulated category $\C$ with $\widehat{\mathcal{U}}=\widecheck{\mathcal{V}}=\C$. As a application, we also obtain a characterization of silting subcategories on the stable category $\underline{\C}$.

The paper is organized as follows. In Section 2, we recall the definition of an extriangulated categorry and outline some basic properties that will be used later. In Section 3, we study the notion of $\X$-resolution dimensions and give some relationships between the relative projective dimension and the $\X$-resolution dimension. Moreover, we focus our attention on the notion of $\X$-injective cogenerators for $\X$ and establish the existence of $\X$-precovers and $\widehat{\W}$-preenvelopes. If $\X$ is closed under extensions, we also obtain that $\widehat{\X}$ is closed under extensions, hence it is an extriangulated category, which is essential for our main result. In Section 4, we define hereditary cotorsion pairs on extriangulated categories with enough projectives and injectives. If $\C$ is a Frobenius extriangulated category, we establish a bijective correspondence between hereditary cotorsion pairs $(\mathcal{U}, \mathcal{V})$ on the extriangulated category $\C$ with $\widehat{\mathcal{U}}=\widecheck{\mathcal{V}}=\C$ and that of certain specially precovering classes. Finally, in Corollary \ref{corl},  we give a characterization of silting subcategories on the stable category $\underline{\C}$ if $\C$ is a Frobenius extriangulated category.

\section{\bf Preliminaries}
Throughout this paper, $\C$ denotes an additive category, by the term $``subcategory"$ we always mean a full additive subcategory of an additive category closed under isomorphisms and direct summands. We denote by ${\mathcal{\C}}(A, B)$ the set of morphisms from $A$ to $B$ in $\C$.

Let $\X$ and $\mathcal{Y}$ be two subcategories of $\C$, a morphism $f: X\rightarrow C$ in $\C$ is said to be an $\X$-precover of $C$ if $X\in\X$ and ${\C}(X', f): {\C}(X', X)\rightarrow {\C}(X', C)$ is surjective, $\forall X'\in\X$. If any $C\in \mathcal{Y}$ admits an $\X$-precover, then $\X$ is called a precovering class in $\mathcal{Y}$. By dualizing the definition above, we get the notion of an $\X$-preenvelope of $C$ and a preenveloping class in $\mathcal{Y}$, for details, see \cite{AR,EJ2}.

Let us briefly recall some definitions and basic properties of extriangulated categories from \cite{NP}.

 \begin{definition}\cite[Definition 2.1]{NP} {\rm Suppose that $\mathcal{C}$ is equipped with an additive bifunctor $$\mathbb{E}: \mathcal{C}^{op}\times \mathcal{C}\rightarrow {\rm Ab},$$  where ${\rm Ab}$ is the category of abelian groups. For any objects $A, C\in\mathcal{C}$, an element $\delta\in \mathbb{E}(C,A)$ is called an $\mathbb{E}$-extension. Thus formally, an $\mathbb{E}$-extension is a triple $(A,\delta, C)$. For any $A, C\in\mathcal{C}$, the zero element $0\in\mathbb{E}(C, A)$ is called the split $\mathbb{E}$-extension.

 Let $\delta\in \mathbb{E}(C, A)$ be any $\mathbb{E}$-extension. By the functoriality, for any $a\in\mathcal{C}(A, A')$ and $c\in\mathcal{C}(C', C)$, we have $\mathbb{E}$-extensions
 \begin{center}$\mathbb{E}(C, a)(\delta)\in\mathbb{E}(C, A')$ ~and~ $\mathbb{E}(c, A)(\delta)\in\mathbb{E}(C', A)$.\end{center}
 We abbreviately denote them by $a_*\delta$ and $c^*\delta$. In this terminology, we have $$\mathbb{E}(c, a)(\delta)=c^*a_*\delta=a_*c^*\delta$$ in $\mathbb{E}(C', A')$.}
 \end{definition}

 \begin{definition}\cite[Definition 2.3]{NP} {\rm Let $\delta\in\mathbb{E}(C, A)$ and $\delta'\in\mathbb{E}(C', A')$ be any pair of $\mathbb{E}$-extensions. A morphism $(a, c): \delta\rightarrow \delta'$ of $\mathbb{E}$-extensions is a pair of morphisms $a\in\mathcal{C}(A, A')$ and $c\in\mathcal{C}(C, C')$ in $\mathcal{C}$ satisfying the equality $$a_*\delta=c^*\delta'.$$
 We simply denote it as $(a, c): \delta\rightarrow \delta'$.}
 \end{definition}

 \begin{definition}\cite[Definition 2.6]{NP} {\rm Let $\delta=(A, \delta, C)$ and $\delta'=(A', \delta', C')$ be any pair of $\mathbb{E}$-extensions. Let
 \begin{center}$\xymatrix{C\ar[r]^{\iota_C\ \ \ }&C\oplus C'&C'\ar[l]_{\qquad\iota_{C'}}}$
 and
 $\xymatrix{A&A\oplus A'\ar[l]_{ p_A\;\;}\ar[r]^{\quad p_{A'} }&A'}$\end{center}
be coproduct and product in $\mathcal{C}$, respectively. Remark that, by the additivity of $\mathbb{E}$, we have a natural isomorphism
\begin{center} $\mathbb{E}(C\oplus C', A\oplus A')\simeq\mathbb{E}(C, A)\oplus\mathbb{E}(C, A')
\oplus\mathbb{E}(C', A)\oplus\mathbb{E}(C', A')$.\end{center}
Let $\delta\oplus \delta'\in\mathbb{E}(C\oplus C', A\oplus A')$ be the element corresponding to $(\delta, 0, 0, \delta')$ through this isomorphism. This is the unique element which satisfies
\begin{center}$\mathbb{E}(\iota_C, p_A)(\delta\oplus \delta')=\delta,\ \mathbb{E}(\iota_C, p_{A'})(\delta\oplus \delta')=0, \ \mathbb{E}(\iota_{C'}, p_A)(\delta\oplus \delta')=0,\ \mathbb{E}(\iota_{C'}, p_{A'})(\delta\oplus \delta')=\delta'$.\end{center}}
 \end{definition}

 \begin{definition}\cite[Definition 2.7]{NP} {\rm Let $A, C\in\mathcal{C}$ be any pair of objects. Two sequences of morphisms in $\mathcal{C}$
 \begin{center} $\xymatrix{A\ar[r]^x&B\ar[r]^y&C}$ and $\xymatrix{A\ar[r]^{x'}&B'\ar[r]^{y'}&C}$\end{center}
 are said to be equivalent if there exists an isomorphism $b\in\mathcal{C}(B, B')$ which makes the following diagram commutative
 $$\small\xymatrix@C=1.2em@R=0.5cm{&&B\ar[dd]^b_\simeq\ar[drr]^y&&\\
 A\ar[urr]^x\ar[drr]_{x'}&&&&C\\
 &&B'\ar[urr]_{y'}}$$
 We denote the equivalence class of $\xymatrix{A\ar[r]^x&B\ar[r]^y&C}$ by $\xymatrix{[A\ar[r]^x&B\ar[r]^y&C]}$.}
 \end{definition}

 \begin{definition}\cite[Definition 2.8]{NP} {\rm
 \emph{(1)} For any $A, C\in\mathcal{C}$, we denote
 $$0=[A\xrightarrow{~\tiny\begin{bmatrix}1\\0\end{bmatrix}~}A\oplus C\xrightarrow{\tiny\begin{bmatrix}0&1\end{bmatrix}}C]$$

\emph{ (2)} For any two classes $\xymatrix{[A\ar[r]^x&B\ar[r]^y&C]}$ and $\xymatrix{[A'\ar[r]^{x'}&B'\ar[r]^{y'}&C']}$, we denote
\begin{center} $\xymatrix{[A\ar[r]^x&B\ar[r]^y&C]}\oplus$$\xymatrix{[A'\ar[r]^{x'}&B'\ar[r]^{y'}&C']}=$$\xymatrix{[A\oplus A'\ar[r]^{x\oplus x'}&B\oplus B'\ar[r]^{y\oplus y'}&C\oplus C']}$.\end{center}}
 \end{definition}

 \begin{definition}\cite[Definition 2.9]{NP} {\rm
  Let $\mathfrak{s}$ be a correspondence which associates an equivalence class $$\mathfrak{s}(\delta)=\xymatrix@C=0.8cm{[A\ar[r]^x
 &B\ar[r]^y&C]}$$ to any $\mathbb{E}$-extension $\delta\in\mathbb{E}(C, A)$. This $\mathfrak{s}$ is called a {\it realization} of $\mathbb{E}$, if it satisfies
 the following condition $(\star)$. In this case, we say that the sequence $\xymatrix{A\ar[r]^x&B\ar[r]^y&C}$ realizes $\delta$, whenever it satisfies $\mathfrak{s}(\delta)=$$\xymatrix{[A\ar[r]^x &B\ar[r]^y&C]}$.

 $(\star)$ Let $\delta\in\mathbb{E}(C, A)$ and $\delta'\in\mathbb{E}(C', A')$ be any pair of $\mathbb{E}$-extensions, with
 \begin{center} $\mathfrak{s}(\delta)=$$\xymatrix{[A\ar[r]^x &B\ar[r]^y&C]}$ and $\mathfrak{s}(\delta')=$$\xymatrix{[A'\ar[r]^{x'} &B'\ar[r]^{y'}&C']}$.\end{center} Then, for any morphism $(a, c): \delta\rightarrow \delta'$, there exists $b\in\mathcal{C}(B, B')$ which makes the following diagram commutative
 $$\xymatrix{A\ar[r]^{x}\ar[d]_{a}&B\ar[r]^{y}\ar[d]_{b}&C\ar[d]_{c}\\
 A'\ar[r]^{x'}&B'\ar[r]^{y'}&C'.}$$
 In the above situation, we say that the triplet $(a, b, c)$ realizes $(a, c)$.}
 \end{definition}

 \begin{definition} \cite[Definition 2.10]{NP} {\rm
 Let $\mathcal{C}, \mathbb{E}$ be as above. A realization of $\mathbb{E}$ is said to be {\it additive}, if it satisfies the following conditions.

 \emph{(i)} For any $A, C\in\mathcal{C}$, the split $\mathbb{E}$-extension $0\in\mathbb{E}(C, A)$ satisfies $\mathfrak{s}(0)=0.$

\emph{(ii)} For any pair of $\mathbb{E}$-extensions $\delta\in\mathbb{E}(C, A)$ and $\delta'\in\mathbb{E}(C', A')$, we have \begin{center}$\mathfrak{s}(\delta\oplus\delta')=\mathfrak{s}(\delta)\oplus\mathfrak{s}(\delta')$.\end{center}}
 \end{definition}

 \begin{definition} \cite[Definition 2.12]{NP} {\rm A triplet $(\mathcal{C}, \mathbb{E}, \mathfrak{s})$ is called an {\it extriangulated category} if it satisfies the following conditions.

 \emph{(ET1)} $\mathbb{E}: \mathcal{C}^{op}\times \mathcal{C}\rightarrow \rm{Ab}$ is an additive bifunctor.

 \emph{(ET2)} $\mathfrak{s}$ is an additive realization of $\mathbb{E}$.

\emph{(ET3)} Let $\delta\in\mathbb{E}(C, A)$ and $\delta'\in\mathbb{E}(C', A')$ be any pair of $\mathbb{E}$-extensions, realized as
 \begin{center} $\mathfrak{s}(\delta)=$$\xymatrix{[A\ar[r]^x &B\ar[r]^y&C]}$ and $\mathfrak{s}(\delta')=$$\xymatrix{[A'\ar[r]^{x'} &B'\ar[r]^{y'}&C']}$.\end{center} For any commutative square
 $$\xymatrix{A\ar[r]^{x}\ar[d]_{a}&B\ar[r]^{y}\ar[d]_{b}&C\\
 A'\ar[r]^{x'}&B'\ar[r]^{y'}&C'.}$$
 in $\mathcal{C}$, there exists a  morphism $(a, c): \delta\rightarrow \delta'$ which is realized by (a, b, c).

 \emph{(ET3)}$^{\rm op}$ Let $\delta\in\mathbb{E}(C, A)$ and $\delta'\in\mathbb{E}(C', A')$ be any pair of $\mathbb{E}$-extensions, realized by
 \begin{center} $\xymatrix{
   A \ar[r]^{x} & B \ar[r]^{y} & C}$ and $\xymatrix{
     A' \ar[r]^{x'} & B'\ar[r]^{y'} & C'  }$\end{center}
 respectively. For any commutative square
 $$\xymatrix{
   A \ar[r]^{x} & B \ar[d]_{b} \ar[r]^{y} & C \ar[d]^{c} \\
   A' \ar[r]^{x'} & B' \ar[r]^{y'} & C'  }$$
   in $\mathcal{C}$, there exists a  morphism $(a, c): \delta\rightarrow \delta'$ which is realized by (a, b, c).

\emph{(ET4)} Let $\delta\in\mathbb{E}(D,A)$ and $\delta'\in\mathbb{E}(F, B)$ be $\mathbb{E}$-extensions realized by
 \begin{center} $\xymatrix{A\ar[r]^f&B\ar[r]^{f'}&D}$ and $\xymatrix{B\ar[r]^g&C\ar[r]^{g'}&F}$\end{center}
 respectively. Then there exist an object $E\in\mathcal{C}$, a commutative diagram
 $$\xymatrix{A\ar[r]^f\ar@{=}[d]&B\ar[r]^{f'}\ar[d]_g&D\ar[d]^d\\
A\ar[r]^h&C\ar[r]^{h'}\ar[d]_{g'}&E\ar[d]^e\\
&F\ar@{=}[r]&F
}$$
in $\mathcal{C}$, and an $\mathbb{E}$-extension $\delta^{''}\in\mathbb{E}(E, A)$ realized by $\xymatrix{A\ar[r]^h&C\ar[r]^{h'}&E}$,
which satisfy the following compatibilities.

\emph{(i)} $\xymatrix{D\ar[r]^d&E\ar[r]^{e}&F}$ realizes $f'_*\delta'$,

\emph{(ii)} $d^*\delta^{''}=\delta$,

\emph{(iii)} $f_*\delta^{''}=e^*\delta'$.

\emph{(ET4)}$^{\rm op}$ Let $\delta\in\E(B, D)$ and $\delta'\in\E(C, F)$ be $\mathbb{E}$-extensions realized by
 \begin{center} $\xymatrix{D\ar[r]^{f'}&A\ar[r]^{f}&B}$ and $\xymatrix{F\ar[r]^{g'}&B\ar[r]^{g}&C}$\end{center}
 respectively. Then there exist an object $E\in\mathcal{C}$, a commutative diagram
 $$\xymatrix{D\ar[r]^d\ar@{=}[d]&E\ar[r]^{e}\ar[d]_{h'}&F\ar[d]^{g'}\\
D\ar[r]^{f'}&A\ar[r]^{f}\ar[d]_{h}&B\ar[d]^g\\
&C\ar@{=}[r]&C
}$$
in $\mathcal{C}$, and an $\mathbb{E}$-extension $\delta^{''}\in\mathbb{E}(C, E)$ realized by $\xymatrix{E\ar[r]^{h'}&A\ar[r]^h&C}$,
which satisfy the following compatibilities.

\emph{(i)} $\xymatrix{D\ar[r]^d&E\ar[r]^{e}&F}$ realizes $g'^*\delta$,

\emph{(ii)} $\delta'=e_*\delta^{''}$,

\emph{(iii)} $d_*\delta=g^*\delta^{''}$.
} \end{definition}

For an extriangulated category $\mathcal{C}$, we use the following notation:

$\bullet$ A sequence $\xymatrix{
    A \ar[r]^{x} & B \ar[r]^{y} & C   }$ is called a conflation if it realizes some $\mathbb{E}$-extension $\delta\in\mathbb{E}(C, A)$.

$\bullet$ A morphism $f\in\mathcal{C}(A, B)$ is called an inflation if it admits some conflation $\xymatrix{
    A \ar[r]^{f} & B \ar[r] & C }$.

$\bullet$ A morphism $f\in\mathcal{C}(A, B)$ is called a deflation if it admits some conflation $\xymatrix{
    K \ar[r]& A \ar[r]^{f} & B   }$.

$\bullet$ If a conflation $\xymatrix{A \ar[r]^{x} & B \ar[r]^{y} & C   }$ realizes $\delta\in\mathbb{E}(C, A)$, we call the pair $(\xymatrix{A\ar[r]^{x} & B \ar[r]^{y} & C   }, \delta)$ an $\mathbb{E}$-triangle, and write it in the following way.
$$\xymatrix{A \ar[r]^{x} & B \ar[r]^{y} & C \ar@{-->}[r]^{\delta}&  }$$
We usually do not write this ``$\delta$" if it is not used in the argument.

$\bullet$ Given an $\mathbb{E}$-triangle $\xymatrix{A \ar[r]^{x} & B \ar[r]^{y} & C \ar@{-->}[r]^{\delta}&  }$, we call $A$ the CoCone of $y:\xymatrix{B \ar[r]& C}$, and denote it by $CoCone(\xymatrix{B \ar[r]& C})$, or $CoCone(y)$; we call $C$ the Cone of $x:\xymatrix{A \ar[r]& B}$, and denote it by $Cone(\xymatrix{A \ar[r]& B})$, or $Cone(x)$.

$\bullet$ Let$\xymatrix{A\ar[r]^{x}&B\ar[r]^{y}&C\ar@{-->}[r]^{\delta}&}$ and $\xymatrix{A'\ar[r]^{x'}&B'\ar[r]^{y'}&C'\ar@{-->}[r]^{\delta'}&}$ be any pair of $\mathbb{E}$-triangles. If a triplet $(a, b, c)$ realizes $(a, c): \delta\rightarrow\delta'$, then we write it as
$$\xymatrix{A\ar[r]^{x}\ar[d]_{a}&B\ar[r]^{y}\ar[d]_{b}&C\ar[d]_{c}\ar@{-->}[r]^{\delta}&\\
 A'\ar[r]^{x'}&B'\ar[r]^{y'}&C'\ar@{-->}[r]^{\delta'}&}$$
and call $(a, b, c)$ a morphism of $\mathbb{E}$-triangles.

$\bullet$ A subcategory $\mathcal{T}$ of $\mathcal{C}$ is called extension-closed if $\mathcal{T}$ is closed under extensions, i.e. for any $\mathbb{E}$-triangle $\xymatrix{A\ar[r]^{x}&B\ar[r]^{y}&C\ar@{-->}[r]^{\delta}&}$ with $A,C\in\mathcal{T}$, we have $B\in\mathcal{T}$.

\begin{Ex} {\rm(1) Exact category $\mathcal{B}$ can be viewed as an extriangulated category. For the definition and basic properties of an exact category, see \cite{B}. In fact, a biadditive functor $\mathbb{E}:={\oext^{1}_{\mathcal{B}} :\mathcal{B}^{\rm op}\times\mathcal{B}\rightarrow Ab.}$ Let $A, C\in\mathcal{B}$ be any pair of objects. Define $\oext^{1}_{\mathcal{B}} (C, A)$ to be the collection of all equivalence classes of short exact sequences of the form $\xymatrix{A\ar[r]^{x}&B\ar[r]^{y}&C}$. We denote the equivalence class by $[\xymatrix{A\ar[r]^{x}&B\ar[r]^{y}&C}]$ as before. For any $\delta=[\xymatrix{A\ar[r]^{x}&B\ar[r]^{y}&C}]\in {\rm Ext}^{1}_{\mathcal{B}} (C, A)$, define the realization $\mathfrak{s}(\delta)$ of $[\xymatrix{A\ar[r]^{x}&B\ar[r]^{y}&C}]$ to be $\delta$ itself. For more detils, see \cite[Example 2.13]{NP}.

(2) Let $\mathcal{C}$ be an triangulated category with shift functor [1]. Put $\mathbb{E}:=\mathcal{C}(-,-[1])$. For any $\delta\in\mathbb{E}(C, A)=\mathcal{C}(C, A[1])$, take a triangle
$$\xymatrix{A\ar[r]^{x} & B \ar[r]^{y} & C\ar[r]^{\delta} & A[1]}$$
and define as $\mathfrak{s}(\delta)=[\xymatrix{A\ar[r]^{x} & B \ar[r]^{y} & C}]$. Then $(\mathcal{C},\mathbb{E},\mathfrak{s})$ is an extriangulated category. It is easy to see that extension closed subcategories of triangulated categories are also extriangulated categories. For more details, see \cite[Proposition 3.22]{NP}.

(3) Let $\mathcal{C}$ be an extriangulated category, and $\mathcal{J}$ a subcategory of $\mathcal{C}$. If $\mathcal{J}\subseteq Proj(\mathcal{C})\bigcap Inj(\mathcal{C})$, where $Proj(\mathcal{C})$ is the full category of projective objects in $\mathcal{C}$ and $Inj(\mathcal{C})$ is the full category of injective objects in $\mathcal{C}$, then $\mathcal{C}/\mathcal{J}$ is an extriangulated category. This construction gives extriangulated categories which are neither exact nor triangulated in general. For more details, see \cite[ Proposition 3.30]{NP}
 }\end{Ex}

\begin{lem}\label{lem1} \emph{\cite[Proposition 3.15]{NP}} Let $(\mathcal{C}, \mathbb{E},\mathfrak{s})$ be an extriangulated category. Then the following hold.

\emph{(1)} Let $C$ be any object, and let $\xymatrix@C=2em{A_1\ar[r]^{x_1}&B_1\ar[r]^{y_1}&C\ar@{-->}[r]^{\delta_1}&}$ and $\xymatrix@C=2em{A_2\ar[r]^{x_2}&B_2\ar[r]^{y_2}&C\ar@{-->}[r]^{\delta_2}&}$ be any pair of $\mathbb{E}$-triangles. Then there is a commutative diagram
in $\mathcal{C}$
$$\xymatrix{
    & A_2\ar[d]_{m_2} \ar@{=}[r] & A_2 \ar[d]^{x_2} \\
  A_1 \ar@{=}[d] \ar[r]^{m_1} & M \ar[d]_{e_2} \ar[r]^{e_1} & B_2\ar[d]^{y_2} \\
  A_1 \ar[r]^{x_1} & B_1\ar[r]^{y_1} & C   }
  $$
  which satisfies $\mathfrak{s}(y^*_2\delta_1)=\xymatrix@C=2em{[A_1\ar[r]^{m_1}&M\ar[r]^{e_1}&B_2]}$ and
  $\mathfrak{s}(y^*_1\delta_2)=\xymatrix@C=2em{[A_2\ar[r]^{m_2}&M\ar[r]^{e_2}&B_1]}$.

  \emph{(2)} Let $A$ be any object, and let $\xymatrix@C=2em{A\ar[r]^{x_1}&B_1\ar[r]^{y_1}&C_1\ar@{-->}[r]^{\delta_1}&}$ and $\xymatrix@C=2em{A\ar[r]^{x_2}&B_2\ar[r]^{y_2}&C_2\ar@{-->}[r]^{\delta_2}&}$ be any pair of $\mathbb{E}$-triangles. Then there is a commutative diagram
in $\mathcal{C}$
$$\xymatrix{
     A\ar[d]_{x_2} \ar[r]^{x_1} & B_1 \ar[d]^{m_2}\ar[r]^{y_1}&C_1\ar@{=}[d] \\
  B_2 \ar[d]_{y_2} \ar[r]^{m_1} & M \ar[d]^{e_2} \ar[r]^{e_1} & C_1 \\
  C_2 \ar@{=}[r] & C_2 &   }
  $$
  which satisfies $\mathfrak{s}(x_{2*}\delta_1)=\xymatrix@C=2em{[B_2\ar[r]^{m_1}&M\ar[r]^{e_1}&C_1]}$ and $\mathfrak{s}(x_{1*}\delta_2)=\xymatrix@C=2em{[B_1\ar[r]^{m_2}&M\ar[r]^{e_2}&C_2]}$.
\end{lem}
Assume that $(\mathcal{C}, \mathbb{E}, \mathfrak{s})$ is an extriangulated category. By Yoneda's Lemma, any $\mathbb{E}$-extension $\delta\in \mathbb{E}(C, A)$ induces  natural transformations
\begin{center} $\delta_\sharp: \mathcal{C}(-, C)\Rightarrow \mathbb{E}(-, A)$ and $\delta^\sharp: \mathcal{C}(A, -)\Rightarrow \mathbb{E}(C, -)$.\end{center}
For any $X\in\mathcal{C}$, these $(\delta_\sharp)_X$ and $\delta^\sharp_X$ are given as follows:

(1) $(\delta_\sharp)_X: \mathcal{C}(X, C)\Rightarrow \mathbb{E}(X, A); ~f\mapsto f^*\delta.$

(2) $\delta^\sharp_X: \mathcal{C}(A, X)\Rightarrow \mathbb{E}(C, X); ~g\mapsto g_*\delta.$
\begin{lem}\label{lem2} {\rm \cite[Corollary 3.12]{NP}}  Let $(\mathcal{C}, \mathbb{E}, \mathfrak{s})$ be an extriangulated category, and $$\xymatrix@C=2em{A\ar[r]^{x}&B\ar[r]^{y}&C\ar@{-->}[r]^{\delta}&}$$ an $\mathbb{E}$-triangle. Then we have the following long exact sequences:

$\xymatrix@C=1cm{\mathcal{C}(C, -)\ar[r]^{\mathcal{C}(y, -)}&\mathcal{C}(B, -)\ar[r]^{\mathcal{C}(x, -)}&\mathcal{C}(A, -)\ar[r]^{\delta^\sharp}&\mathbb{E}(C, -)\ar[r]^{\mathbb{E}(y, -)}&\mathbb{E}(B, -)\ar[r]^{\mathbb{E}(x, -)}&\mathbb{E}(A, -)};$

$\xymatrix@C=1cm{\mathcal{C}(-, A)\ar[r]^{\mathcal{C}(-, x)}&\mathcal{C}(-, B)\ar[r]^{\mathcal{C}(-, y)}&\mathcal{C}(-, C)\ar[r]^{\delta_\sharp}&\mathbb{E}(-, A)\ar[r]^{\mathbb{E}(-, x)}&\mathbb{E}(-, B)\ar[r]^{\mathbb{E}(-, y)}&\mathbb{E}(-, C)}.$

\end{lem}
The following lemma was given in \cite[Proposition 1.20]{LN}, which is another version of \cite[Corollary 3.16]{NP}.
\begin{lem}\label{lem3}{\rm \cite[Proposition 1.20]{LN}}
Let $\xymatrix{A\ar[r]^{x}&B\ar[r]^{y}&C\ar@{-->}[r]^{\delta}&}$ be an $\mathbb{E}$-triangle and $f: A\rightarrow B$ any morphism, and let $\xymatrix{D\ar[r]^{d}&E\ar[r]^{e}&C\ar@{-->}[r]^{f_{*}\delta}&}$ be any $\mathbb{E}$-triangle realizing $f_{*}\delta$. Then there is a morphism $g$ which gives a morphisms of $\mathbb{E}$-triangles$$\xymatrix@C=2em{A\ar[r]^x\ar[d]^f&B\ar[r]^y\ar[d]^g&C\ar@{-->}[r]^{\delta}\ar@{=}[d]&\\
  D\ar[r]^d&E\ar[r]^e&C\ar@{-->}[r]^{f_{*}(\delta)}&
  }$$
  Moreover, $\xymatrix@C=1,2cm{A\ar[r]^{\tiny\begin{bmatrix}-f\\x\end{bmatrix}\ \ \ }&D\oplus B\ar[r]^{\tiny\ \ \begin{bmatrix}d&g\end{bmatrix}}&E\ar@{-->}[r]^{e^*\delta}&}$ is an $\mathbb{E}$-triangle.
\end{lem}
We recall some concepts from \cite{NP}. Let $\mathcal{C},\mathbb{E}$ be as above. An object $P\in \mathcal{C}$ is called projective if it satisfies the following condition.

$\bullet$ For any $\E$-triangle $\xymatrix{A\ar[r]^{x}&B\ar[r]^{y}&C\ar@{-->}[r]^{\delta}&}$ and any morphism $c\in \mathcal{C}(P, C)$, there exists $b\in \mathcal{C}(P, B)$ satisfying $y\circ b=c$.

Injective objects are defined dually.

We denote the subcategory consisting of projective objects in $\C$ by $Proj(\C)$. Dually, the subcategory of injective objects in $\C$ is denoted by $Inj(\C)$.

$\bullet$ We say $\mathcal{C}$ has enough projectives (enough injectives, resp.) if for any object $C\in\C$($A\in \mathcal{C}$, resp.), there exists an $\E$-triangle
$$\xymatrix{A\ar[r]^{x}&P\ar[r]^{y}&C\ar@{-->}[r]^{\delta}&}(\xymatrix{A\ar[r]^{x}&I\ar[r]^{y}&C\ar@{-->}[r]^{\delta}&})$$satisfying $P\in Proj(\mathcal{C})$~$(I\in Inj(\mathcal{C}), resp.)$.

In this case, $A$ is called the syzygy of $C$ ($C$ is called the cosyzygy of $A$, resp.) and is denoted by $\Omega(C)$~$(\Sigma(A), resp.)$.

$\bullet$ $\mathcal{C}$ is said to be $Frobenius$ if $\mathcal{C}$ has enough projectives and injectives and if moreover the projectives coincide with the injectives. In this case one has the quotient category $\underline{\mathcal{C}}$ of $\mathcal{C}$ by projectives, which is a triangulated category by \cite {NP}. We refer to this as the stable category of $\C$.

\begin{rem} \rm{(1) If $(\mathcal{C}, \mathbb{E}, \mathfrak{s})$ is an exact category, then the definitions of enough projectives and enough injectives agree with the usual definitions.

(2) If $(\mathcal{C}, \mathbb{E}, \mathfrak{s})$ is a triangulated category, then $Proj(\mathcal{C})$ and $Inj(\mathcal{C})$ consist of zero objects. Moreover $(\mathcal{C}, \mathbb{E}, \mathfrak{s})$ is Frobenius as an extriangulated category.
}\end{rem}

Suppose $\mathcal{C}$ is an extriangulated category with enough projectives and injectives. For a subcategory $\mathcal{B}\subseteq \mathcal{C}$, put $\Omega^{0}\mathcal{B}=\mathcal{B}$, and for $i>0$ we define $\Omega^{i}\mathcal{B}$ inductively to be the subcategory consisting of syzygies of objects in $\Omega^{i-1}$, i.e.
$$\Omega^{i}\mathcal{B}=\Omega(\Omega^{i-1}\mathcal{B}).$$
We call $\Omega^{i}\mathcal{B}$ the $i$-th syzygy of $\mathcal{B}$. Dually we define the $i$-th cosyzygy $\Sigma^{i}\mathcal{B}$ by $\Sigma^{0}\mathcal{B}=\mathcal{B}$ and $\Sigma^{i}\mathcal{B}=\Sigma(\Sigma^{i-1}\mathcal{B})$ for $i>0.$

In \cite{LN} the authors defined higher extension groups in an extriangulated category having enough projectives and injectives as $\E^{i+1}(X, Y)\cong\E(X, \Sigma^{i}Y)\cong\E(\Omega^{i}X, Y)$ for $i\geq 0$, and they showed the following result:

\begin{lem}\label{lem4} Let $\xymatrix{A\ar[r]^{x}&B\ar[r]^{y}&C\ar@{-->}[r]^{\delta}&}$ be an $\E$-triangle. For any object $X\in \mathcal{B}$, there are long exact sequences
$$\xymatrix@C=0.5cm{\cdots\ar[r] &\E^{i}(X, A)\ar[r]^{x_{*}}&\E^{i}(X, B)\ar[r]^{y_{*}}&\E^{i}(X, C)\ar[r]&\E^{i+1}(X, A)\ar[r]^{x_{*}}&\E^{i+1}(X, B)\ar[r]^{y_{*}}&\cdots(i\geq1)},$$

$$\xymatrix@C=0.5cm{\cdots\ar[r] &\E^{i}(C, X)\ar[r]^{y^{*}}&\E^{i}(B, X)\ar[r]^{x^{*}}&\E^{i}(A, X)\ar[r]&\E^{i+1}(C, X)\ar[r]^{y^{*}}&\E^{i+1}(B, X)\ar[r]^{x^{*}}&\cdots(i\geq1)}.$$
\end{lem}

An $\E$-triangle sequence in $\mathcal{C}$ \cite{ZZ1} is displayed as a sequence
$$\xymatrix{\cdots\ar[r]&X_{n+1}\ar[r]^{d_{n+1}}&X_{n}\ar[r]^{d_{n}}&X_{n-1}\ar[r]&\cdots&}$$
over $\mathcal{C}$ such that for any $n$, there are $\E$-triangles $\xymatrix{K_{n+1}\ar[r]^{g_{n}}&X_{n}\ar[r]^{f_{n}}&K_{n}\ar@{-->}[r]^{\delta^{n}}&}$ and the differential $d_{n}=g_{n-1}f_{n}$.

From now on to the end of the paper, we always suppose that extriangulated category $\C$ has enough projectives and injectives.

\section{\bf Relative Homological Dimensions}

 Let  $\mathcal{X}$ be a subcategory of $\mathcal{C}$. The symbol $\widehat{\X_{n}}$ ($\widecheck{\X_{n}}$, resp.) denotes the subcategory of objects $A\in \C$ such that there exits an $\E$-triangle sequence $X_{n}\rightarrow X_{n-1}\rightarrow\cdots\rightarrow X_{0}\rightarrow C $~$(C\rightarrow X_{0}\rightarrow\cdots\rightarrow X_{n-1}\rightarrow X_{n}, $ resp.$)$ with each $X_{i}\in \mathcal{X}$. We denote by $\widehat{X}$ ( $\widecheck{X}$, resp.) the union of all $\widehat{\X_{n}}$ ($\widecheck{\X_{n}}$, resp.) for some nonnegative $n$. That is to say $\widehat{\X}=\bigcup\limits_{n=0}^\infty \mathcal{X}_n, \widecheck{\X}=\bigcup\limits_{n=0}^\infty \widecheck{\X_{n}}$.

\begin{df}
(1) For any $C\in \mathcal{C}$, the $\mathcal{X}$-resolution dimension of $C$ is
\begin{center}
resdim$_{\mathcal{X}}(C)$:=min$\{n\in\mathbb{N}:C\in\widehat{\mathcal{X}_{n}}\}$.
\end{center}
If $C\not\in\widehat{\X_{n}}$ for any $n\in \mathbb{N}$, then resdim$_{\X}(C)=\infty$.
Dually, we also have the $\mathcal{X}$-coresolution dimension of $C$ denoted by coresdim$_{\mathcal{X}}(C)$.

(2) For any subcategory $\mathcal{Y}$ of $\mathcal{C}$, we set
\begin{center}
resdim$_{\mathcal{X}}(\mathcal{Y})$:=sup$\{$resdim$_{\mathcal{X}}(M): M\in\mathcal{Y}\}$.
\end{center}

Dually, we also have coresdim$_{\mathcal{X}}(\mathcal{Y})$.

(3) The $\mathcal{X}$-projective dimension of $C$ is
\begin{center}
pd$_{\mathcal{X}}(C)$:=min$\{n\in \mathbb{N}:\E^{i}(C, -)|_{\mathcal{X}}=0, \forall i>n \}$.
\end{center}

(4) The $\mathcal{X}$-injective dimension of $\mathcal{C}$ is
\begin{center}
id$_{\mathcal{X}}(C)$:=min$\{n\in \mathbb{N}:\E^{i}(-, C)|_{\mathcal{X}}=0, \forall i>n \}$.
\end{center}

(5) For any subcategory $\mathcal{Y}$ of $\mathcal{C}$, we set
\begin{center}
pd$_{\mathcal{X}}(\mathcal{Y})$:=sup$\{$pd$_{\mathcal{X}}(M): M\in\mathcal{Y}\}$ $~$and$~$ id$_{\mathcal{X}}(\mathcal{Y})$:=sup$\{$id$_{\mathcal{X}}(M): M\in\mathcal{Y}\}$.
\end{center}

\end{df}

\begin{lem}\label{lem5} Let $\mathcal{X}$ and $\mathcal{Y}$ be subcategories of $\mathcal{C}$. Then $pd_{\mathcal{X}}(\mathcal{Y})=id_{\mathcal{Y}}(\mathcal{X})$.
Furthermore, for any $\E$-triangle $\xymatrix{A\ar[r]^{x}&B\ar[r]^{y}&C\ar@{-->}[r]^{\delta}&}$ in $\mathcal{C}$, we have

\emph{(1)} $id_{\mathcal{X}}(B)\leq max\{id_{\mathcal{X}}(A), id_{\mathcal{X}}(C)\}$;

\emph{(2)} $id_{\mathcal{X}}(A)\leq max\{id_{\mathcal{X}}(B), id_{\mathcal{X}}(C)+1\}$;

\emph{(3)} $id_{\mathcal{X}}(C)\leq max\{id_{\mathcal{X}}(B), id_{\mathcal{X}}(A)-1\}$.
\end{lem}

\begin{proof}
It is straightforward.
\end{proof}

For a subcategory $\mathcal{X}$ of $\mathcal{C}$, define $\mathcal{X}^{\perp}=\{Y\in\mathcal{C}|\E^{i}(X, Y)=0, \forall i\geq 1, X \in \mathcal{X}\}$. Similarly, we can define $^{\perp}\mathcal{X}.$ Now we give a relationship between the relative projective dimension and the resolution dimension.

\begin{thm}\label{thm1}
Let $\mathcal{X}$ and $\mathcal{Y}$ be subcategories of $\mathcal{C}$. Then, the following statements hold.

(1) pd$_{\mathcal{X}}(L)\leq$ pd$_{\mathcal{X}}(\mathcal{Y})$+resdim$_{\mathcal{Y}}(L)$, $\forall L \in \mathcal{C}$.

(2) If $\mathcal{Y}\subseteq \mathcal{X}\bigcap$ $^{\perp}\mathcal{X}$, then pd$_{\mathcal{X}}(L)$=resdim$_{\mathcal{Y}}(L)$, $\forall L \in \widehat{\mathcal{Y}}$.

\end{thm}
\begin{proof}
(1) Let $d$:=resdim$_{\mathcal{Y}}(L)$ and $\alpha$:=pd$_{\mathcal{X}}(\mathcal{Y})$. We may assume that $d$ and $\alpha$ are finite. We prove (1) by induction on $d$. If $d=0$, it follows that $L\in \mathcal{Y}$, hence (1) holds in this case.

Assume $d\geq 1$. So we have an $\E$-triangle $$\xymatrix{K\ar[r]^{x}&Y\ar[r]^{y}&L\ar@{-->}[r]^{\delta}&}$$ in $\mathcal{C}$ with $Y\in \mathcal{Y}$ and resdim$_{\mathcal{Y}}(K)=d-1$. Applying ${\rm Hom}_\C(-, M)$, with $M\in \X$, to the $\E$-triangle $\delta$, we get an exact sequence $\E^{i-1}(K, M)\rightarrow \E^{i}(L, M)\rightarrow \E^{i}(Y, M)$. By induction, we know that $pd_{\mathcal{X}}(K)\leq \alpha+d-1$. Therefore $\E^{i}(L, M)=0$ for $i>\alpha+d$, and so $pd_{\mathcal{X}}(L)\leq \alpha+d$.

(2) Let $\mathcal{Y}\subseteq \mathcal{X}\bigcap$ $^{\perp}\mathcal{X}$. Consider $L \in \widehat{\mathcal{Y}}$ and let $d$=resdim$_{\mathcal{Y}}(L)$. Since $pd_{\mathcal{X}}(\mathcal{Y})=0$, it follows that pd$_{\mathcal{X}}(L)\leq$resdim$_{\mathcal{Y}}(L)=d$ by $(1)$. We prove, by induction on $d$, that the equality given in $(2)$ holds. For $d=0$, it is obvious.

Suppose that $d=1$. Then we have an $\E$-triangle $$\xymatrix{Y_{1}\ar[r]^{x}&Y_{0}\ar[r]^{y}&L\ar@{-->}[r]^{\delta}&}$$ in $\mathcal{C}$ with $Y_{i}\in \mathcal{Y}, i=1, 2$. If pd$_{\mathcal{X}}(L)=0$, then $L\in$ $~^{\perp}\mathcal{X}$. Since $\mathcal{Y}\subseteq \mathcal{X}$, $\E(L, Y_{1})=0$, therefore the $\E$-triangle $\delta$ splits giving us that $L\in \mathcal{Y}$, which is a contradiction as $d=1$. So pd$_{X}(L)>0$ proving (2) for $d=1$.

Assume now that $d\geq 2$. Thus we have an $\E$triangle $$\xymatrix{K\ar[r]^{d}&Y\ar[r]^{e}&L\ar@{-->}[r]^{\theta}&}$$ in $\mathcal{C}$ with $Y\in\mathcal{Y}$, resdim$_{\mathcal{Y}}(K)=d-1$. Hence pd$_{\mathcal{X}}(K)=d-1$ by inductive hypothesis. For any $X\in \mathcal{X}$, there is an exact sequence
$$\E^{d-1}(Y, X)\rightarrow \E^{d-1}(K, X)\rightarrow \E^{d}(L, X).$$
If pd$_{\mathcal{X}}(L)\leq d-1$, then $\E^{d-1}(K, X)$=0 contradicting that pd$_{\mathcal{X}}(K)=d-1$. This means that pd$_{\mathcal{X}}(L)>d-1$; proving (2).
\end{proof}

Now, we begin to focus our attention on pairs $(\X, \W)$ of subcategories of $\C$ and study the relationship between $\X$-injective cogenerators for $\X$ and $\widehat{\X}$.
\begin{df}
Let $\mathcal{X}$ and $\mathcal{W}$ be two subcategories of $\mathcal{C}$. We say that

(1) $\mathcal{W}$ is a cogenerator for $\mathcal{X}$, if $\mathcal{W}\subseteq\mathcal{X}$ and for each object $X\in \mathcal{X}$, there exists an $\E$-triangle $\xymatrix{X\ar[r]^{x}&W\ar[r]^{y}&X'\ar@{-->}[r]^{\delta}&}$. The term generator is defined dually.

(2) $\mathcal{W}$ is $\mathcal{X}$-injective if id$_{\mathcal{X}}(\mathcal{W})$=0. The term $\mathcal{X}$-projective is defined dually.

(3) $\mathcal{W}$ is an $\mathcal{X}$-injective cogenetator for $\X$ if $\mathcal{W}$ is a cogenerator for $\mathcal{X}$ and id$_{\mathcal{X}}(\mathcal{W})$=0. The term $\mathcal{X}$-projective generator for $\X$ is defined dually.
\end{df}

In the following, let $\X$ and $\W$ be two subcategories of $\C$ such that $\W\subseteq\X$.

\begin{lem}\label{lem6}
Suppose that $\mathcal{X}$ is an extension closed subcategory of $\mathcal{C}$. Consider two $\E$-triangles
\begin{center}
$\xymatrix{N\ar[r]^{a}&X_{1}\ar[r]^{b}&D\ar@{-->}[r]^{\delta}&} and   \xymatrix{D\ar[r]^{c}&X_{0}\ar[r]^{d}&M\ar@{-->}[r]^{\theta}&}$
\end{center}
 in $\mathcal{C}$ with $X_{0}, X_{1}\in \mathcal{X}$. If $\mathcal{W}$ is a cogenetator for $\mathcal{X}$, then there exist two $\E$-triangles
\begin{center}
$\xymatrix{N\ar[r]^{}&W_{1}\ar[r]^{}&D'\ar@{-->}[r]^{}&} and \xymatrix{D'\ar[r]^{}&X_{0}'\ar[r]^{}&M\ar@{-->}[r]^{}&}$
\end{center}
 with $W_{1}\in \mathcal{W}$ and $X_{0}'\in\mathcal{X}$.
\end{lem}

\begin{proof}
 We have an $\E$-triangle $\xymatrix{X_{1}\ar[r]^{}&W_{1}\ar[r]^{}&X_{1}'\ar@{-->}[r]^{}&}$ with $W_{1}\in\W$ and $X_{1}'\in\X$ as $\W$ is a cogenerator for $\X$. By $(ET4)$, we obtain a commutative diagram in $\C$
\begin{center}
$\xymatrix{N\ar[r]^{}\ar@{=}[d]&X_{1}\ar[r]^{}\ar[d]_{}&D\ar[d]^{}&\\
N\ar[r]^{}&W_{1}\ar[r]^{}\ar[d]_{}&D'\ar[d]^{}&\\
&X_{1}'\ar@{=}[r]&X_{1}'.&&&}$
\end{center}
 By Lemma \ref{lem1} (2), we have the following commutative diagram in $\C$
$$\xymatrix{
     D\ar[d]_{} \ar[r]^{} & X_0 \ar[d]^{}\ar[r]^{}&M\ar@{=}[d] \\
  D' \ar[d]_{} \ar[r]^{} & X_{0}' \ar[d]^{} \ar[r]^{} & M \\
  X_{1}' \ar@{=}[r] & X_{1}'.}
  $$ Since $\X$ is closed under extensions, it follows that $X_{0}'\in \X$. The second rows in the above two diagrams are desired $\E$-triangles.
\end{proof}

\begin{lem}\label{lem7}
Suppose $\X$ is closed under extensions and $\W$ is a cogenerator for $\X$. Then for any $X\in\X$ and nonnegative integer $n$, $C\in\widehat{\X_{n}}$ if and only if there exists an $\E$-triangle
$$W_{n}\rightarrow\cdots\rightarrow W_{2}\rightarrow W_1\rightarrow X_0\rightarrow C$$ with $X_{0}\in \X$ and $W_{i}\in\W$ for $1\leq i\leq n$.
\end{lem}

\begin{proof}
 The ``~if~" part is trivial. We prove the  ``~only if~" part by induction on $n$. If $n=1$, then there exists an $\E$-triangles
\begin{center}
  $\xymatrix{X_1\ar[r]^{}&X_{0}\ar[r]^{}&C\ar@{-->}[r]^{}&}$ and $\xymatrix{0\ar[r]^{}&X_{1}\ar[r]^{}&X_{1}\ar@{-->}[r]^{}&}$ with $X_{1}, X_0\in \X$.
\end{center}
   By Lemma \ref{lem6}, we abtain an $\E$-trangle $\xymatrix{W_{1}\ar[r]^{}&X_{0}'\ar[r]^{}&C\ar@{-->}[r]^{}&}$ that is desired.

 Suppose now $n\geq 2$. Then we have an $\E$-trangle $$\xymatrix{K\ar[r]^{}&X_{0}'\ar[r]^{}&C\ar@{-->}[r]^{}&}$$
with $X_0 \in \X$ and $K \in \widehat{\X_{n-1}}$. By the induction, there exists an $\E$-trangle sequence
 $$W_n\rightarrow \cdots\rightarrow W_{2}\stackrel{f} \longrightarrow X_1\rightarrow K$$ with $X_1 \in \X$ and
 $W_i \in \W$ for $2\leq i\leq n$.
  Hence we have an $\E$-trangle sequence $$W_{n-1}\rightarrow \cdots \rightarrow W_2 \stackrel{f_1}\longrightarrow K'$$ and a $\E$-trangle $$\xymatrix{K'\ar[r]^{f_2}&X_1\ar[r]&K\ar@{-->}[r]&}$$ with $f_2f_1=f$.
 Applying Lemma \ref{lem6}(1) to $\E$-triangles
 \begin{center}
 $\xymatrix{K'\ar[r]^{f_2}&X_1\ar[r]&K\ar@{-->}[r]&}$ and $\xymatrix{K\ar[r]^{}&X_{0}'\ar[r]^{}&C\ar@{-->}[r]^{}&}$,
 \end{center}
  we obtain two $\E$-triangles
\begin{center}
  $\xymatrix{K'\ar[r]^{}&W_1\ar[r]&K''\ar@{-->}[r]&}$ and $\xymatrix{K''\ar[r]^{}&X_{0}\ar[r]^{}&C\ar@{-->}[r]^{}&}$.
\end{center}
   Therefore, we have a $\E$-triangle sequence
 $$W_{n}\rightarrow\cdots\rightarrow W_{2}\rightarrow W_1\rightarrow X_0\rightarrow C$$ with $X_{0}\in \X$ and $W_{i}\in\W$ for $1\leq i\leq n$.
\end{proof}

The following theorem shows that any object in $\widehat{\X}$ admits two $\E$-triangles: one giving rise to an $\X$-precover and the other to a $\widehat{\W}$-preenvelope, which generalizes \cite[Theorem 5.4]{MSSS1}.
\begin{thm}\label{thm2}
Suppose $\X$ is closed under extensions and $\W$ is a cogenerator for $\X$. Consider the following conditions:

(1) $C$ is in $\widehat{\X_n}.$

(2) There exists an $\E$-triangle
$$\xymatrix{Y_{C}\ar[r]^{}&X_{C}\ar[r]^{\varphi_{C}}&C\ar@{-->}[r]^{\delta}&}$$
with $X_{C}\in \X$ and $Y_{C}\in \widehat{\W_{n-1}}.$

(3) There exists an $\E$-triangle
$$\xymatrix{C\ar[r]^{\psi^{C}}&Y^{C}\ar[r]^{}&X^{C}\ar@{-->}[r]^{\theta}&}$$
with $X^{C}\in \X$ and $Y^{C}\in \widehat{\W_{n}}.$

Then, $(1)\Leftrightarrow(2)\Rightarrow(3)$. If $\X$ is also closed under CoCones, then $(3)\Rightarrow(2)$, and hence all three conditions are equivalent. If $\W$ is $\X$-injective, then $\varphi_{C}$ is an $\X$-precover of $C$ and $\psi^{C}$ is a $\widehat{\W}$-preenvelope of $C$.

\end{thm}
\begin{proof}
$(1)\Leftrightarrow(2)$ follows from Lemma \ref{lem7}.

$(2)\Rightarrow(3)$ Since $X_{C}\in\X$ and $\W$ is a cogenerator for $\X$, we have an $\E$-triangle
$$\xymatrix{X_{C}\ar[r]^{x}&W\ar[r]^{y}&X'\ar@{-->}[r]^{}&}$$ with $W\in\W$ and $X'\in\X$.
By $(ET4)$, we obtain a commutative diagram
$$\xymatrix{Y_{C}\ar[r]^{}\ar@{=}[d]&X_{C}\ar[r]^{}\ar[d]_{}&C\ar[d]^{}&\\
Y_{C}\ar[r]^{}&W\ar[r]^{}\ar[d]_{}&Y^{C}\ar[d]^{}&\\
&X'\ar@{=}[r]&X'.&&&}$$
From the second row, it follows that $Y^{C}\in \widehat{\W_{n}}$. Hence the third column is the desired one.

Suppose $\X$ is also closed under CoCones. Since $Y^{C}\in \widehat{\W_{n}}$, we have an $\E$-triangle
$$\xymatrix{Y_{C}\ar[r]^{x}&W\ar[r]^{y}&Y^{C}\ar@{-->}[r]^{}&}$$
in $\C$ with $W\in \W$ and $Y_{C}\in \widehat{W_{n-1}}$. By $(ET4)^{op}$, we obtain a commutative diagram
$$\xymatrix{Y_{C}\ar[r]^{}\ar@{=}[d]&X_{C}\ar[r]^{}\ar[d]_{}&C\ar[d]^{}&\\
Y_{C}\ar[r]^{}&W\ar[r]^{}\ar[d]_{}&Y^{C}\ar[d]^{}&\\
&X^{C}\ar@{=}[r]&X^{C}.&&&}$$ Since $\X$ is closed under CoCones, it follows that $X_{C}\in \X$. Hence the first row is the desired $\E$-triangle.

Assume $\W$ is $\X$-injective. Applying $\mathcal{C}(X,-)$ to the $\E$-triangle

$$\xymatrix{Y_{C}\ar[r]^{}&X_{C}\ar[r]^{\varphi_{C}}&C\ar@{-->}[r]^{\delta}&},$$
we have an exact sequence 
$$\xymatrix{\mathcal{C}(X, X_{C})\ar[rr]^{\mathcal{C}(X, \varphi_{C})}&&{\mathcal{C}}(X, C)\ar[r]^{}&\E(X, Y_{C}).&&}$$ Since $\E(X, Y_{C})=0$, it follows that ${\rm Hom}_{\mathcal{C}}(X, \varphi_{C})$ is an epimorphism, hence $\varphi_{C}$ is an $\X$-precover of $C$. Similarly, we can prove that $\psi^{C}$ is a $\widehat{\W}$-preenvelope of $C$.
\end{proof}

\begin{lem}\label{lem8}
Let $\X$ and $\mathcal{Y}$ be subcategories of $\C$. Then
id$_{\mathcal{Y}}(\mathcal{\widehat{\X}})$=id$_{\mathcal{Y}}(\mathcal{\X}).$

\end{lem}
\begin{proof}
Since $\X\subseteq\widehat{\X}$, it follows that id$_{\mathcal{Y}}(\mathcal{\X})$$\leq$id$_{\mathcal{Y}}(\mathcal{\widehat{\X}}).$  It is enough to show that id$_{\mathcal{Y}}(\mathcal{\widehat{\X}})$$\leq$id$_{\mathcal{Y}}(\mathcal{\X}).$ We may assume that $\alpha$:=id$_{\mathcal{Y}}(\mathcal{\X})$ is finite. Let $C\in \widehat{\X}$. We prove that, by induction on $n=$resdim$_{\mathcal{X}}(C)$. If $n=0$, then $C\in \X$, there is nothing to prove.

Let $n\geq1$. Then we have an $\E$-triangle
$$\xymatrix{K\ar[r]^{}&X\ar[r]^{}&C\ar@{-->}[r]^{\delta}&}$$ with $X\in\X, K\in\widehat{\X_{n-1}}$, id$_{\mathcal{Y}}(K)\leq \alpha$ by inductive hypothesis. For any $Y\in \mathcal{Y}$, applying ${\rm Hom}_{\mathcal{C}}(X,-)$ to the $\E$-triangle $\delta$, we obtain an exact sequence
$$\E^{i}(Y, K)\rightarrow\E^{i}(Y, X)\rightarrow\E^{i}(Y, C)\rightarrow\E^{i+1}(Y, K).$$ Therefore $\E^{i}(Y, C)=0$ for $i>\alpha$ as id$_{\mathcal{Y}}(K)\leq\alpha.$ So we get id$_{\mathcal{Y}}(\mathcal{\widehat{\X}})$$\leq$id$_{\mathcal{Y}}(\mathcal{\X}).$
\end{proof}

The following result shows that there is a unique $\X$-injective cogenerator for $\X$ (in case it exists).
\begin{prop}\label{prop1}
Let $\X$ and $\W$ be two subcategories of $\C$ such that $\W$ is $\X$-injecive. Then the following statements hold.

(1) $\widehat{\W}$ is $\X$-injective.

(2) If $\W$ is a cogenerator for $\X$, then $\W=\X\bigcap\X^{\perp}=\X\bigcap\widehat{\W}$.

(3) If $\W$ is a cogenerator for $\X$, then $\widehat{\W}=\widehat{\X}\bigcap\X^{\perp}$.
\end{prop}
\begin{proof}
(1) It follows from the Lemma \ref{lem8}.

(2) Let $X\in \X\bigcap\X^{\perp}$. We have an $\E$-triangle
$$\xymatrix{X\ar[r]^{}&W\ar[r]^{y}&X'\ar@{-->}[r]^{\theta}&}$$ with $X'\in \X$ and $W\in\W$. Moreover $X\in \X^{\bot}$ implies that the $\E$-triangle $\theta$ splits and so $X\in \W$. On the other hand, it is easy to see $\W\subseteq\X\bigcap\widehat{\W}$. Since id$_{\X}(\widehat{\W})=0$, it follows that $\X\bigcap\widehat{\W}\subseteq \X\bigcap\X^{\perp}$. Hence $\W=\X\bigcap\widehat{\W}.$

(3) Let $C\in \widehat{\X}\bigcap\X^{\perp}$. We have an $\E$-triangle
$$\xymatrix{Y_{C}\ar[r]^{}&X_{C}\ar[r]^{}&C\ar@{-->}[r]^{\theta}&}$$ with $Y_{C}\in \widehat{\W}$ and $X_{C}\in\X$ by Theorem \ref{thm2}. Since $Y_{C}$ and $C$ are in $\X^{\perp}$, it follows that $X_{C}\in \X\bigcap\X^{\perp}$. Hence $X_{C}\in \W$ by $(2)$, implying $\widehat{\X}\bigcap\X^{\bot}\subseteq\widehat{\W}$. On the other hand, it is obvious that $\widehat{\W}\subseteq\widehat{\X}\bigcap\X^{\perp}.$ Therefore $\widehat{\W}=\widehat{\X}\bigcap\X^{\perp}$.
\end{proof}

\begin{prop}\label{prop2}
Let $\X$ be closed under extensions such that $\W$ is an $\X$-injective cogenerator for $\X$. Then $\widehat{\X}$ is closed under extensions, and hence an extriangulated category.

\end{prop}

\begin{proof}
Suppose $\xymatrix{A\ar[r]^{d}&B\ar[r]^{e}&C\ar@{-->}[r]^{\delta}&}$ is an $\E$-triangle with $A$ and $C$ in $\widehat{\X}$. Proceed by induction on $n:=$resdim$_{\X}(C)$. Assume $n=0$, which means that $C$ is in $\X$. As $A$ is in $\widehat{\X}$, there exists an $\E$-triangle
$$\xymatrix{Y_{A}\ar[r]^{}&X_{A}\ar[r]^{p}&A\ar@{-->}[r]^{\xi}&}$$ with $X_{A}\in \X$ and $Y_{A}\in\widehat{\W_{n-1}}$ by Theorem \ref{thm2}. Since $\widehat{\W}\subseteq\X^{\perp}$, it follows that $\E(C, p): \E(C, X_{A})\rightarrow \E(C, A)$ is an isomorphism. Hence we obtain a commutative diagram
$$\xymatrix@C=2em{X_{A}\ar[r]^x\ar[d]^p&Z\ar[r]^y\ar[d]^g&C\ar@{-->}[r]^{\theta}\ar@{=}[d]&\\
  A\ar[r]^d&B\ar[r]^e&C\ar@{-->}[r]^{\delta}&
  }$$ with $p_{*}\theta=\delta$. By $(ET4)$, we have the following commutative diagram
  $$\xymatrix{Y_{A}\ar[r]^{}\ar@{=}[d]&X_{A}\ar[r]^{p}\ar[d]_{x}&A\ar[d]^{d}&\\
Y_{A}\ar[r]^{}&Z\ar[r]^{}\ar[d]_{y}&B\ar[d]^{e}&\\
&C\ar@{=}[r]&C.&&&}$$
Since $X_{A}$ and $C$ are in $\X$, the object $Z\in \X$ as $\X$ is closed under extensions. Note that $Y_{A}\in \widehat{\W}$, hence in $\widehat{\X}$, it follows that $B\in \widehat{\X}$.

Assume $n>0$ and let $\xymatrix{L\ar[r]^{}&X_{0}\ar[r]^{y}&C\ar@{-->}[r]^{}&}$ be an $\E$-triangle with resdim$_{\X}(L)=n-1$. By Lemma \ref{lem1}, we have a commutative diagram
$$\xymatrix{
    & L\ar[d]_{} \ar@{=}[r] & L \ar[d]^{} \\
  A \ar@{=}[d] \ar[r]^{} & V' \ar[d]_{} \ar[r]^{} & X_{0}\ar[d]^{} \\
  A \ar[r]^{} & B\ar[r]^{} & C.   }
  $$
Since $X_{0}\in\X$, it follows that $\E(X_{0},p): \E(X_{0}, X_{A})\rightarrow \E(X_{0}, A)$ is an isomorphism, using (ET4), we obtain a commutative diagram
  $$\xymatrix{Y_{A}\ar[r]^{}\ar@{=}[d]&X_{A}\ar[r]^{p}\ar[d]_{}&A\ar[d]^{}&\\
Y_{A}\ar[r]^{}&V\ar[r]^{}\ar[d]_{}&V'\ar[d]^{}&\\
&X_{0}\ar@{=}[r]&X_{0}.&&&}$$
Using $(ET4)^{op}$ again, we also have the following commutative diagram
 $$\xymatrix{Y_{A}\ar[r]^{}\ar@{=}[d]&U\ar[r]^{p}\ar[d]_{x}&L\ar[d]^{d}&\\
Y_{A}\ar[r]^{}&V\ar[r]^{}\ar[d]_{}&V'\ar[d]^{}&\\
&B\ar@{=}[r]&B.&&&}$$
 By inductive hypothesis, $U$ is in $\widehat{\X}$. Since $X_{A}$ and $X_{0}$ are in $\X$, $V\in \X$, as $\X$ is closed under extensions. It follows that $B\in \widehat{\X}$. Hence $\widehat{\X}$ is also an extriangulated category by \cite[Remark 2.18]{NP}.
\end{proof}

\begin{lem}\label{lem9}
Let $\X$ be closed under extensions and CoCones such that $\W$ is an $\X$-injective cogenerator for $\X$. Giving an $\E$-triangle
$\xymatrix{K\ar[r]^{x}&X\ar[r]^{y}&C\ar@{-->}[r]^{\delta}&}$ with $X\in\X$. Then $C\in \widehat{\X}$ if and only if $K\in \widehat{\X}$.
\end{lem}
\begin{proof}
By definition, if $K\in \widehat{\X}$, so is $C \in\widehat{\X}$. Now assume that $C\in \widehat{\X}$. By Theorem \ref{thm2}, we obtain a commutative diagram
$$\xymatrix@C=2em{K\ar[r]^x\ar[d]^f&X\ar[r]^y\ar[d]^g&C\ar@{-->}[r]^{\delta}\ar@{=}[d]&\\
  Y_{C}\ar[r]^l&X_{C}\ar[r]^p&C\ar@{-->}[r]^{f_{*}\delta}&
  }$$ with $X_{C}\in \X$ and $Y_{C}\in \widehat{\W}$. Since  $Y_{C}\in \widehat{\W}$, there exists an $\E$-triangle
$$\xymatrix{L\ar[r]^{m}&W\ar[r]^{e}&Y_{C}\ar@{-->}[r]^{\delta'}&}$$ with $W\in\W$.
 We also have a commutative diagram
$$\xymatrix@C=2em{L\ar[r]^u\ar@{=}[d]&V\ar[r]^v\ar[d]^h&K\ar@{-->}[r]^{f^{*}\delta'}\ar[d]^{f}&\\
  L\ar[r]^m&W\ar[r]^e&Y_{C}\ar@{-->}[r]^{\delta'}&
  .}$$
By the dual of Lemma \ref{lem3}, there exists an $\E$-triangle $\xymatrix@C=1,2cm{V\ar[r]^{\tiny\begin{bmatrix}-v\\h\end{bmatrix}\ \ \ }&K\oplus W\ar[r]^{\tiny\ \ \begin{bmatrix}f&e\end{bmatrix}}&Y_{C}\ar@{-->}[r]^{u_{*}\delta'}&}$. By $(ET4)$, we obtain a commutative diagram
$$\xymatrix@C=3em{
  V\ar[r]^{\tiny\begin{bmatrix}-v\\h\end{bmatrix}} \ar@{=}[d] &K\oplus W \ar[r]^{\tiny\begin{bmatrix}f&e\end{bmatrix}} \ar[d]^{\tiny\begin{bmatrix}x&0\\0&1\end{bmatrix}}& Y_{C} \ar[d]^{} \ar@{-->}[r]^{u_{*}\delta'} &  \\
  V\ar[r]^{} & X\oplus W \ar[r]^{} \ar[d]^{\tiny\begin{bmatrix}y&0\end{bmatrix}} & H\ar@{-->}[r]^{} \ar[d]^{} &\\
  & C \ar@{-->}[d]^{\delta\oplus 0} \ar@{=}[r] &C\ar@{-->}[d]^{\tiny\begin{bmatrix}f&e\end{bmatrix}_*(\delta\oplus 0)}&. \\
  &&&}\eqno{(*)}
$$
Since $\tiny\begin{bmatrix}f&e\end{bmatrix}_*(\delta\oplus0)=f_{*}\delta$, we obtain a commutative diagram
$$\xymatrix@C=2em{Y_{C}\ar[r]^{}\ar@{=}[d]&H\ar[r]^{}\ar[d]^{}&C\ar@{-->}[r]^{f_{*}\delta}\ar@{=}[d]&\\
  Y_{C}\ar[r]^{l}&X_{0}\ar[r]^{p}&C\ar@{-->}[r]^{f_{*}\delta}&
  .}$$ It follows that $H$ is isomorphic to $X_{C}$. Hence we can replace $H$ by $X_{C}$ in $(*)$ and have the commutative diagram
  $$\xymatrix@C=3em{
  V\ar[r]^{\tiny\begin{bmatrix}-v\\h\end{bmatrix}} \ar@{=}[d] &K\oplus W \ar[r]^{\tiny\begin{bmatrix}f&e\end{bmatrix}} \ar[d]^{\tiny\begin{bmatrix}x&0\\0&1\end{bmatrix}}& Y_{C} \ar[d]^{l} \ar@{-->}[r]^{u_{*}\delta'} &  \\
  V\ar[r]^{} & X\oplus W \ar[r]^{} \ar[d]^{\tiny\begin{bmatrix}y&0\end{bmatrix}} & X_{C}\ar@{-->}[r]^{} \ar[d]^{p} &\\
  & C \ar@{-->}[d]^{\delta\oplus 0} \ar@{=}[r] &C\ar@{-->}[d]^{f_{*}\delta}&. \\
  &&&}
$$ Since $W$ and $X$ are both in $\X$, it follows that $V$ is in $\X$ as $\X$ is closed under Cocones. By Proposition \ref{prop2}, $K\oplus W\in \widehat{\X}.$ So there exists an $\E$-triangle
$$\xymatrix{Y_{K\oplus W}\ar[r]^{}&X_{K\oplus W}\ar[r]^{}&K\oplus W\ar@{-->}[r]^{}&}$$ with $Y_{K\oplus W}\in\widehat{\W}, X_{K\oplus W}\in\X$ by Theorem \ref{thm2}. By $(ET4)^{op}$, there is a commutative diagram
$$\xymatrix{Y_{K\oplus W}\ar[r]^{}\ar@{=}[d]&Z\ar[r]^{}\ar[d]_{}&W\ar[d]^{}&\\
Y_{K\oplus W}\ar[r]^{}&X_{K\oplus W}\ar[r]^{}\ar[d]_{}&K\oplus W\ar[d]^{}&\\
&K\ar@{=}[r]&K&&&}$$ implying $Z\in \widehat{\X}$ as $Y_{K\oplus W}$ and $W$ are in $\widehat{\X}$. Hence $K\in \widehat{\X}$. This completes the proof of the lemma.
\end{proof}

\begin{prop}\label{prop3}
Let $\X$ be closed under extensions such that $\W$ is an $\X$-injective cogenerator for $\X$. Then pd$_{\widehat{\W}}(C)=$pd$_{\W}(C)=$resdim$_{\X}(C)$ for any $C\in \widehat{\X}$.
\end{prop}
\begin{proof}
Let $C\in \widehat{\X}$, pd$_{\W}(C)=$id$_{\{C\}}(\W)=$id$_{\{C\}}(\widehat{\W})$ by Lemma \ref{lem5} and \ref{lem8}. To prove the last equality, we proceed by induction on $n=:$resdim$_{\X}(C)$. If $n=0$, then $C\in \X$ and pd$_{\W}(C)$=resdim$_{\X}(C)=0$

Let $n=1$. Then we have an $\E$-triangle$$\xymatrix{Y_{C}\ar[r]^{}&X_{C}\ar[r]^{}&C\ar@{-->}[r]^{\delta}&}$$ with $X_{C}\in\X$ and $Y_{C}\in \W$ by Theorem \ref{thm2}. We claim that pd$_{\W}(C)>0$. Indeed, suppose pd$_{\W}(C)=0$, it follows that the $\E$-triangle $\delta$ splits. Hence $C\in\X$ contradicting that resdim$_{\X}(C)=1$. Hence pd$_{\W}(C)=1$ by Theorem \ref{thm1}.

Let $n\geq 2$. From Theorem \ref{thm1}, we have pd$_{\W}(C)\leq$resdim$_{\X}(C)=n$ as pd$_{\W}(\X)=0$. Then it is enough to prove that $\E^{n}(C, W)\neq0$ for some $W\in \W$. Consider an $\E$-triangle$$\xymatrix{K\ar[r]^{}&X_{0}\ar[r]^{}&C\ar@{-->}[r]^{\xi}&}$$ with $X_{0}\in \X$ and resdim$_{\X}(K)=n-1$. By inductive hypothesis pd$_{\W}(K)=n-1$. Applying ${\rm Hom }_{\mathcal{C}}(-, W)$ to the $\E$-triangle $\xi$ with $W\in \W$, we have an exact sequence
$$\E^{n-1}(X_{0},W)\rightarrow\E^{n-1}(K, W)\rightarrow\E^{n}(C, W).$$
Suppose that $\E^{n}(C, -)|_{\W}=0$, then $\E^{n-1}(K, -)|_{\W}=0$ as id$_{\X}(\W)=0$ and $n\geq 2;$ contradicting that pd$_{\W}(K)=n-1$.
\end{proof}

Now we give a characterizations of $\X$-resolution dimensions of objects in $\widehat{\X}$.
\begin{thm}\label{thm3}
Let $\X$ be closed under extensions and CoCones such that $\W$ is an $\X$-injective cogenerator for $\X$. The following are equivalent for any $C\in \widehat{\X}$ and nonnegative integer $n$.

(1) resdim$_{\X}(C)\leq n$.

(2) If $U\rightarrow X_{n-1}\rightarrow\cdots\rightarrow X_{0}\rightarrow C$ is an $\E$-triangle sequence with $X_{i}\in \X$ for $0\leq i\leq n-1$, then $U\in \X$.

(3) $\E^{n+i}(C, Y)=0$ for any object $Y\in \widehat{\W}$ and $i\geq1$.

(4) $\E^{n+i}(C, W)=0$ for any object $W\in \W$ and $i\geq1$.

(5) $\E^{n+1}(C, W)=0$ for any object $W\in \W$.

\end{thm}

\begin{proof}
By Proposition \ref{prop3}, we have $(1)\Rightarrow (3)\Leftrightarrow(4)\Rightarrow(5)$. $(2)\Rightarrow(1)$ is trivial.

$(5)\Rightarrow(1)$ Since $C\in \widehat{\X}$, there is an $\E$-triangle sequence
$$U\rightarrow X_{n-1}\rightarrow\cdots\rightarrow X_{0}\rightarrow C$$ with $X_{i}\in \X$ for $0\leq i\leq n-1$. We also have $\E(U, \W)\cong \E^{n+1}(C, \W)=0$ as $\W$ is $\X$-injective. Note that $U\in \widehat{\X}$ by Lemma \ref{lem9}. Hence there exists an $\E$-triangle
$$\xymatrix{Y_{U}\ar[r]^{}&X_{U}\ar[r]^{}&U\ar@{-->}[r]^{\theta}&}$$ with $X_{U}\in \X$ and $Y_{U}\in \widehat{\W}$ by Theorem \ref{thm2}. It follows that the $\E$-triangle $\theta$ splits as $\E(U, \W)=0$. Hence $U\in \X$.
\end{proof}

\begin{prop}\label{cor1}
Let $\X$ be closed under extensions and CoCones such that $\W$ is an $\X$-injective cogenerator for $\X$. Then $\widehat{\X}$ is closed under direct summands.
\end{prop}
\begin{proof}
Suppose $C_{1}\oplus C_{2}\in \widehat{\X}$. Proceed by induction on $n=:$resdim$_{\X}(C_{1}\oplus C_{2})$. If $n=0$, then $C_{1}$ and $C_{2}$ are in $\X$.

Suppose $n>0$. There is an $\E$-triangle
$$\xymatrix{K\ar[r]^{}&X\ar[r]^{y}&C_{1}\oplus C_{2}\ar@{-->}[r]^{}&}$$
with $X\in \X$ and resdim$_{\X}(K)=n-1$. By $(ET4)^{op}$, we obtain a commutative diagram
$$\xymatrix@C=3em{
  K\ar[r]^{} \ar@{=}[d] &L_{2} \ar[r]^{} \ar[d]^{x_{2}}& C_{1} \ar[d]^{\tiny\begin{bmatrix}1\\0\end{bmatrix}} \ar@{-->}[r]^{} &  \\
  K\ar[r]^{x} & X\ar[r]^{\tiny\begin{bmatrix}y_{1}\\y_{2}\end{bmatrix}} \ar[d]^{y_{2}} & C_{1}\oplus C_{2}\ar@{-->}[r]^{\delta} \ar[d]^{\tiny\begin{bmatrix}0&1\end{bmatrix}} &\\
  & C_{2} \ar@{-->}[d]^{\delta_{2}} \ar@{=}[r] &C_{2}\ar@{-->}[d]^{0}& \\
  &&&}
.$$ Similarly, we can obtain an $\E$-triangle $$\xymatrix{L_{1}\ar[r]^{x_{1}}&X\ar[r]^{y_{1}}&C_{1}\ar@{-->}[r]^{\delta_{1}}&}.$$
Hence there is an $\E$-triangle $$\xymatrix{L_{1}\oplus L_{2}\ar[r]^{\tiny\begin{bmatrix}x_{1}&0\\0&x_{2}\end{bmatrix}}&X\oplus X\ar[r]^{\tiny\begin{bmatrix}y_{1}&0\\0&y_{2}\end{bmatrix}}&C_{1}\oplus C_{2}\ar@{-->}[r]^{\delta_{1}\oplus\delta_{2}}&}.$$ By Lemma \ref{lem9}, $L_{1}\oplus L_{2}\in \widehat{\X}$, and Theorem \ref{thm3} shows that resdim$_{\X}(L_{1}\oplus L_{2})\leq n-1$. By inductive hypothesis, $L_{1}$ and $L_{2}$ are in $\widehat{\X}$. It follows that $C_{1}$ and $C_{2}$ are in $\widehat{\X}$ by Lemma \ref{lem9}
\end{proof}

The following result will play a key role in Section 4.
\begin{prop}\label{cor2}
Let $\X$ be closed under extensions and CoCones such that $\W$ is an $\X$-injective cogenerator for $\X$. For any $\E$-triangle $\xymatrix{A\ar[r]^{x}&B\ar[r]^{y}&C\ar@{-->}[r]^{\delta}&}$, if any two of $A, B$ and $C$ are in $\widehat{\C}$, then the third one is in $\widehat{\X}$.
\end{prop}
\begin{proof}
Since we already knew that $\widehat{\X}$ is closed under extensions by Proposition \ref{prop2}, it suffices to show that if $B\in\widehat{\X}$, then $A\in \widehat{\X}$ if and only if $C\in \widehat{\X}$. We first show that if $A$ and $B$ are in $\widehat{\X}$, then $C\in \widehat{\X}$. Since $B\in \widehat{\X}$, we have an $\E$-triangle
$$\xymatrix{Y_{B}\ar[r]^{}&X_{B}\ar[r]^{}&B\ar@{-->}[r]^{}&}$$ with $X_{B}\in \X, Y_{B}\in\widehat{\W}$. By $(ET4)^{op}$, we obtain a commutative diagram
$$\xymatrix{Y_{B}\ar[r]^{}\ar@{=}[d]&L\ar[r]^{}\ar[d]_{}&A\ar[d]^{}&\\
Y_{B}\ar[r]^{}&X_{B}\ar[r]^{}\ar[d]_{}&B\ar[d]^{}&\\
&C\ar@{=}[r]&C
&&&.}$$ It follows that $L\in \widehat{\X}$ as $A$ and $Y_{B}$ are $\in \widehat{\X}$. Therefore $C\in \widehat{\X}.$

Suppose now $B$ and $C$ are $\widehat{\X}$. By Lemma \ref{lem9}, $L\in\widehat{\X}$. Applying the just established result to $\E$-triangle
$$\xymatrix{Y_{B}\ar[r]^{}&L\ar[r]^{}&A\ar@{-->}[r]^{}&}$$ gives $A\in \widehat{\X}$.
\end{proof}

We say $\X$ is resolving (coresolving, resp.) \cite{ZZ1} if it contains $Proj(\C) ~(Inj(\C), resp.)$, closed under extensions and CoCones (Cones, resp.).

\begin{cor}\label{cor3}
Let $\X$ be a resolving subcategory of $\C$ such that $\W$ is an $\X$-injective cogenerator for $\X$. If $\C$ is a Frobenius extriangulated category, then $\underline{\widehat{\X}}$ is a thick subcategory of $\underline{\C}$.
\end{cor}
\begin{proof}
If $\C$ is a Frobenius extriangulated category, it is easy to see that $\widehat{\X}$ is also a Frobenius extriangulated category by Proposition \ref{cor2}. By \cite[Remark 7.5]{NP} and  Proposition \ref{cor1}, $\underline{\widehat{\X}}$ is a thick subcategory of $\underline{\C}$.
\end{proof}

\section{\bf Cotorsion Pairs in Extriangulated categories}
We recall the definition of a cotorsion pair in an extriangulated category from \cite{NP}.
\begin{definition} {\rm Let $\mathcal{U}$ and $\mathcal{V}$ be subcategories of an extriangulated category $\C$. The pair $(\mathcal{U}, \mathcal{V})$ is called a cotorsion pair if it satisfies the following conditions.

(1) $\E(\mathcal{U}, \mathcal{V})=0.$

(2) For any $C\in \C$, there exists an $\E$-triangle $$\xymatrix{V_{C}\ar[r]^{x}&U_{C}\ar[r]^{y}&C\ar@{-->}[r]^{\delta}&}$$
satisfying $U_{C}\in\mathcal{U}, V_{C}\in\mathcal{V}$.

(3) For any $C\in \C$, there exists an $\E$-triangle $$\xymatrix{C\ar[r]^{f}&V^{C}\ar[r]^{g}&U^{C}\ar@{-->}[r]^{\theta}&}$$
satisfying $U^{C}\in\mathcal{U}, V^{C}\in\mathcal{V}$.
}
\end{definition}

\begin{rem} \rm{Let $(\mathcal{U}, \mathcal{V})$ be a cotorsion pair on an extriangulated category $\C$. Then

(1) $\mathcal{U}$ is a precovering class in $\C$ and $\mathcal{V}$ is a preenveloping class in $\C$.

(2) $C\in \mathcal{U}$ if and only if $\E(C, \mathcal{V})=0.$

(3) $C\in \mathcal{V}$ if and only if $\E(\mathcal{U}, C)=0.$

(4) $\mathcal{U}$ and $\mathcal{V}$ are closed under extensions.

(5) $Proj(\C)\subseteq\mathcal{U}$ and $Inj(\C)\subseteq\mathcal{V}$.
}

\end{rem}

\begin{lem}\label{lem10}
For a cotorsion pair $(\mathcal{U}, \mathcal{V})$ on $\C$, the following conditions are equivalent.

(1) $\mathcal{U}$ is resolving.

(2) $\mathcal{V}$ is coresolving.

(3) id$_{~\mathcal{U}}(\mathcal{V})=0.$
\end{lem}
\begin{proof}
The proof is similar to that of \cite[Lemma 5.24]{GT12}.
\end{proof}

\begin{rem}{\rm Note that hereditary cotorsion pairs can only be defined on an extriangulated category with enough projectives and injectives.
}
\end{rem}

\begin{prop}\label{prop4}
Let $\X$ be closed under extensions and CoCones such that $\W$ is an $\X$-injective cogenerator for $\X$. Then $(\X, \widehat{\W})$ is a cotorsion pair on the extriangulated category $\widehat{\X}$.
\end{prop}
\begin{proof}
 This comes from Theorem \ref{thm2}, Proposition \ref{prop1}, \ref{prop2} and \ref{cor1}.
\end{proof}

\begin{definition}
Let $\HH$ be a subcategory of $\C$. $\HH$ is called specially precovering in $\C$ provided that for any $C\in \C$, there is an $\E$-triangle  $$\xymatrix{K\ar[r]^{}&H\ar[r]^{}&C\ar@{-->}[r]^{}&}$$ with $H\in\HH, \E(\HH, K)=0$.
\end{definition}

\begin{lem}\label{lem11}
Let $\HH$ be a subcategory of $\C$. Suppose that $\HH$ is resolving and specially precovering. Then $\HH\bigcap\HH^{\bot}$ is an $\HH$-injective cogenerator for $\HH$.
\end{lem}
\begin{proof}
Let $H\in\HH$. There is an $\E$-triangle  $$\xymatrix{H\ar[r]^{}&I\ar[r]^{}&X\ar@{-->}[r]^{}&}$$ with $I\in Inj(\C)$. Since $\HH$ is specially precovering, we have an $\E$-triangle $$\xymatrix{K\ar[r]^{}&H'\ar[r]^{}&X\ar@{-->}[r]^{}&}$$ with $H'\in\HH$ and $\E(\HH, K)=0$. By Lemma \ref{lem1}, we have the following commutative diagram
$$\xymatrix{
    & H\ar[d]_{} \ar@{=}[r] & H \ar[d]^{} \\
  K \ar@{=}[d] \ar[r]^{} & M \ar[d]_{} \ar[r]^{} & I\ar[d]^{} \\
  K \ar[r]^{} & H'\ar[r]^{} & X.   }
  $$
Since $\HH$ is closed under extensions, it follows that $M\in \HH$. Note $\E(\HH, I)=\E(\HH, K)=0$. So $\E(\HH, M)=0$. We claim $M\in \HH\bigcap\HH^{\bot}$. Indeed, for any positive integer $n$ and $H'\in\HH$, we have the following $\E$-triangle sequence
$$L\rightarrow P_{n-1}\rightarrow\cdots\rightarrow P_{1}\rightarrow P_{0}\rightarrow H''$$ with $P_{i}\in Proj(\C)$ for $0\leq i\leq n-1$. We have $\E^{n}(H'', M)\cong\E(L, M)=0$ as $\HH$ is resolving and $\E(\HH, M)=0$. Hence $M\in \HH\bigcap\HH^{\bot}$. The second column in the above diagram implies that $\HH\bigcap\HH^{\bot}$ is an $\HH$-injective cogenerator for $\HH$.
\end{proof}




\begin{thm}\label{prop5}
Let $\C$ be a Frobenius extriangulated category.
The assignments
\begin{center}
$(\mathcal{U}, \mathcal{V})\mapsto\mathcal{U}$ $~~$ and $~~$ $\mathcal{H}\mapsto(\mathcal{H}, \widehat{\mathcal{M}})$,
\end{center}
 where $\mathcal{M}=\mathcal{H}\bigcap\mathcal{H}^{\bot}$, give mutually inverse bijections between the following classes:

(1) Hereditary cotorsion pairs $(\mathcal{U}, \mathcal{V})$ on $\C$ with $\widehat{\mathcal{U}}=\C$ and $\widecheck{\mathcal{V}}=\C$.

(2) Subcategories $\mathcal{H}$ of $\C$, which is specially precovering and resolving in $\C$ such that $\widehat{\mathcal{H}}=\C$ and for any $H\in \mathcal{H}$, there exists a postive integer $i\geq 1$ making $\E^{i}(H', H)=0$ for any $H'\in \mathcal{H}$.

\end{thm}

\begin{proof}
Let $(\mathcal{U}, \mathcal{V})$ be a hereditary cotorsion pair on $\C$ with $\widehat{\mathcal{U}}=\C$ and $\widecheck{\mathcal{V}}=\C$. Then $\mathcal{U}$ is precovering and resolving in $\C$ by Remark 4.2 and Lemma \ref{lem10}.

For any $U\in \mathcal{U}, U\in \widecheck{\mathcal{V}}$ by assumption. Hence there exists an $\E$-triangle sequence
$$U\rightarrow V_{0}\rightarrow V_{1}\rightarrow\cdots\rightarrow V_{n}$$ with $V_{i}\in \mathcal{V}$ for $0\leq i\leq n$. We have $\E^{n+1}(U', U)\cong\E(U', V_{n})=0$ for any $U'\in \mathcal{U}$ . Clearly, $n+1$ is the desired $i$.

 Assume now that $\mathcal{H}$ is a subcategory of $\C$ satisfying the conditions in (2). By Lemma \ref{lem11}, $\mathcal{M}=\HH\bigcap\HH^{\bot}$ is an $\mathcal{H}$-injective cogenerator for $\mathcal{H}$. By Proposition \ref{prop1}, $\widehat{\mathcal{M}}=\widehat{\mathcal{H}}\bigcap\mathcal{H}^{\bot}=\mathcal{H}^{\bot}$, where the second equality is due to $\widehat{\mathcal{H}}=\C$. Therefore $(\mathcal{H}, \widehat{\mathcal{M}})$ is a hereditary cotorsion pair on $\C$ by Proposition \ref{prop4}. For any $X\in \C$, we have an $\E$-triangle
$$\xymatrix{X\ar[r]^{}&Y\ar[r]^{}&L\ar@{-->}[r]^{}&}$$ with $Y\in\widehat{\mathcal{M}}$ and $L\in \mathcal{H}$ by Theorem \ref{thm2}. We claim that $L\in \widetilde{\mathcal{M}}:=\widecheck{\hat{\mathcal{M}}}$.  Indeed, note that there exists a postive integer $i$ such that $\E^{i}(H, L)=0$ for any $H\in \mathcal{H}$. If $i=1$, then $L\in \widehat{\mathcal{M}}$, as desired. Assume $i>1$, then $\E^{i}(H, L)=0$ for any $H\in \mathcal{H}$ implies that $\Sigma^{i-1}(L)\in \mathcal{H}^{\bot}$. Meanwhile $\Sigma^{i-1}(L)\in \widehat{\mathcal{H}}$ as $\mathcal{H}$ contains $Inj(\C)$. Hence $\Sigma^{i-1}(L)\in \widehat{\mathcal{H}}\bigcap\mathcal{H}^{\bot}=\widehat{\mathcal{M}}$ by Proposition \ref{prop1}. Since Inj$(\C)\subseteq\widehat{\mathcal{M}}$, it follows that $L\in \widetilde{\mathcal{M}}$. Hence $\C=\widetilde{\mathcal{M}}$. This completes the proof.
\end{proof}

We conclude the paper by the following result, which gives a characterization of silting subcategories on stable categories and also generalizes \cite[Corollary 3.7]{DLWW}. For the convenience of the reader, we recall the definition of silting subcategories.

\begin{definition}\cite[Definition 2.1]{AI} {\rm Let $\mathcal{M}$ be a subcategory of a triangulated category $\mathcal{T}$. $\mathcal{M}$ is called silting if ${\rm{Hom}}_{\mathcal{T}}(\mathcal{M}, \mathcal{M}[\geq1])=0$ and $\mathcal{T}={\rm thick}(\mathcal{M})$, where $\mathcal{T}={\rm thick}(\mathcal{M})$ is the smallest triangulated subcategory of $\mathcal{T}$ which contains $\mathcal{M}$ and is closed under direct summands and isomorphisms}.

\end{definition}
\begin{cor}\label{corl}
Let $\C$ be a Frobenius extriangulated category. The assignments
\begin{center}
 $\underline{\mathcal{M}}\mapsto$$^{\bot}\mathcal{M}$~$~~$ and $~~$~$\mathcal{H}\mapsto\underline{\mathcal{H}\bigcap\mathcal{H}^{\bot}}$
\end{center}
  give mutually inverse bijections between the following classes: 

(1) Silting subcategories $\underline{\mathcal{M}}$ of the stable category $\underline{\C}$.

(2) Subcategories $\mathcal{H}$ of $\C$, which is specially precovering and resolving in $\C$ such that $\widehat{\mathcal{H}}=\C$ and for any $H\in \mathcal{H}$, there exists a postive integer $i\geq 1$ making $\E^{i}(H', H)=0$ for any $H'\in \mathcal{H}$.
\end{cor}
\begin{proof}
Since $\C$ is a Frobenius extriangulated category, it follows that $\underline{\C}$ is a triangulated category by \cite[Corollary 7.4]{NP}.  By \cite[Corollary 3.10]{CZZ}, $(\mathcal{U}, \mathcal{V})$ is a cotorsion pair on $\C$ if and only if $(\mathcal{\underline{U}}, \mathcal{\underline{V}})$ is a cotorsion pair on $\underline{\C}.$
Note that any distinguished triangle $$\xymatrix{A\ar[r]^{\underline{f}} & B \ar[r]^{\underline{g}} & C\ar[r]^{\underline{h}} & A\langle1\rangle}$$ in $\underline{\C}$ is induced by an $\E$-triangle
$$\xymatrix{A\ar[r]^{f}&B\ar[r]^{g}&C\ar@{-->}[r]^{\delta}&}.$$ By \cite[Remark 4.2]{MSSS2}, it is easy to see that $(\mathcal{U}, \mathcal{V})$ is a hereditary cotorsion pair on $\C$ with $\widehat{\mathcal{U}}=\widecheck{\mathcal{V}}=\C$ if and only if $(\underline{\mathcal{U}}, \underline{\mathcal{V}})$ is a bounded co-t-structure on $\underline{\C}$. Hence the Corollary follows from \cite[Corollary 5.9]{MSSS2} and Theorem \ref{prop5}.
\end{proof}

\renewcommand\refname
\end{document}